\documentclass[10pt]{article}
 \usepackage{amsmath,amsthm,amscd}
 \usepackage{xcolor}
 \usepackage[all]{xy}% for diagrams
 	 \usepackage[stix2]{newtxmath} 	    % substitute  stix2-fonts for newtxmath gives nicer "v"
 	 \usepackage[cal=boondoxo]{mathalfa} % mathcal
	  \usepackage{physics} % for italic bold

 \usepackage[all]{xy}% for diagrams
 \usepackage{mathtools}
 
\usepackage[pagebackref]{hyperref} 
 \usepackage{graphicx}
 \usepackage{enumitem}
\newtheorem{thm}{Theorem}[section]
	\newtheorem*{thm*}{Theorem}
	
	\newtheorem*{lemma*}{Lemma}
	
	\newtheorem{prop}[thm]{Proposition}
	\newtheorem*{prop*}{Proposition}
	\newtheorem{corr}[thm]{Corollary}
	\newtheorem*{corr*}{Corollary}
	
	\newtheorem{obs}[thm]{\textbf{Observation}}
\theoremstyle{definition} 
	\newtheorem{dfn}[thm]{Definition}
	
	\newtheorem{modif}[thm]{\textbf{Modification}}
\theoremstyle{remark}

	\newtheorem{rmk}[thm]{\textit{Remark}}
	\newtheorem{exmple}[thm]{Example}
	
\def\half{\frac 12}

 \usepackage[cal=boondoxo]{mathalfa} % mathcal

\newcommand{\bC}{{\mathbb{C}}}

\newcommand{\bR}{{\mathbb{R}}}
\newcommand{\bZ}{{\mathbb{Z}}}
\def\ii{{\mbold i}}
\newcommand{\mbold}[1]{\vb*{#1}}

	\newcommand{\bY}{{\mbold{Y}}}

	\newcommand{\bo}{{\mbold{0}}}
	\newcommand{\bp}{{\mbold{p}}}

	\newcommand{\bx}{{\mbold{x}}}

	\newcommand{\bz}{{\mbold{z}}}
 \newcommand\cL{{\mathcal L}}
\def\mapright#1{\mathop{\vbox{\ialign{
                                ##\crcr
    ${\scriptstyle\hfil\;\;#1\;\;\hfil}$\crcr
 \noalign{\kern2pt\nointerlineskip}
    \rightarrowfill\crcr}}\;}}
    \def\mapleft#1{\mathop{\vbox{\ialign{
                                ##\crcr
    ${\scriptstyle\hfil\;\;#1\;\;\hfil}$\crcr
 \noalign{\kern2pt\nointerlineskip}
    \leftarrowfill\crcr}}\;}}
\def\into{\hookrightarrow}

\def\comp{\raise1pt\hbox{{$\scriptstyle\circ$}}} % composition
\def\del#1#2{\frac{\partial #1}{\partial #2}} % partial differentiation
\renewcommand\setminus{-} % setminus is not skew
\def\set#1{\{ #1\}} 
\def\sett#1#2{\{ #1 \mid  #2 \}}  
\def\jac#1{\mathsf{Jac}_{#1}} % jacobian ring
\def\mf#1{\mathsf F_{#1}} % Milnor fiber of #1
\def\mt#1#2{\mathsf{T}_{#1}{(#2)}} % Milnor tube
\def\omt#1#2{\mathsf{T}^*_{#1}{(#2)}} %open Milnor tube
 \def\cmt#1#2{\widehat{\mathsf{T}^*_{#1}{(#2)}} }% completed Milnor tube
\def\lnk#1{\mathsf L_{#1}} % Link of #1
\def\cyl#1{\mathsf{Cyl}{(#1)}}

\begin{document}

\title{On symplectic aspects of $\mu$-equivalence  of  a class of isolated hypersurface singularities}
\author{Chris Peters}
 \date{ }
\maketitle

\abstract{  A $\mu$-constant deformation of isolated complex hypersurface singularities of dimensions different from 2 preserves the diffeomorphism class of the Milnor fiber. 
If it would be known that   $\mu$-constant deformations  have no vanishing folds 
as defined  in \cite{Oshea2}, this result is even  true in all dimensions since
the members in the family would have a common Milnor ball-radius in the sense of \S~\ref{sec:top}.

In this note a symplectic version of this  is shown which in particular applies to weighted homogeneous polynomial singularities.
%The method of proof also gives    a symplectic version of the well known result that analytically equivalent hypersurface singularities have diffeomorphic Milnor fibers 
%and   links (in all dimensions). 
}

\section{Introduction}
\label{sec:intro}

In recent times a fruitful interplay between the algebra geometric and symplectic aspects of isolated hypersurface singularities took off, 
perhaps spurred by the influential paper \cite{mclean} by  McLean. 
Another striking example of this interplay is the paper  \cite{boba} by de Bobadilla and Pe{\l}ka where they solve the
 Zariski conjecture stating that isolated hypersurface singularities with the same Milnor number have the same 
multiplicity. See also  the simplification  by \cite{auy} which uses symplectic geometry in a different, more effective way.

 This note is just a  modest observation on   deformations of isolated hypersurface singularities with fixed Milnor number
  assuming these  have a  common Milnor  radius, or, equivalently, assuming there are no   so-called  vanishing folds, 
 as  introduced by O'Shea in \cite{Oshea2} and whose definition is recalled in  \S~\ref{ssec:isodefo}. 
 This immediately    implies that the Milnor fibers of the deformations are differentiably isomorphic.
 Since the  occurrence  or absence of vanishing folds seems a hard problem, L\^e and  Ramanujam circumvented this.
 With a  different approach  they showed   
  in  \cite{LeRam} that  the members of a family with constant Milnor number are 
 diffeomorphic, at least in dimensions   different from 2. 
  \par
  Here the symplectic version is investigated under the hypothesis that there are no vanishing folds, 
  which leads to the main result in this note which reads as follows.
  
\begin{thm*}[=Theorem~\ref{thm:MuConstResult} and Corollary~\ref{cor:MuConst}]
 The  completed Milnor fibers of a family  of  isolated (polynomial) hypersurface singularities without vanishing folds 
 are   mutually symplectomorphic,  and the links of the corresponding singularities are mutually contactomorphic.
   This is in particularly true for families of quasi-homogeneous polynomials with constant Milnor numbers.
  \end{thm*}
 
 The notions mentioned in this theorem will be explained in \S~\ref{ssec:bassymp} after having recalled
 the classical results on Milnor fibers of isolated hypersurface singularities   in \S~\ref{sec:top}.   

The   proof  of  Theorem~\ref{thm:MuConstResult} given    in \S~\ref{ssec:muequiv}  depends on results of A.~Keating in \cite[\S9.1]{keat}
which are stated here as Proposition~\ref{prop:keat}\footnote{See  Remark~\ref{rmk:onsympfib}(1)  for
a construction in \cite{boba}  related to  Keating's.}.

A direct generalization of the L\^e--Ramanujam approach  seems hard since, -- as explained in Remark~\ref{rmk:hCobord} --  they make essential use of the $h$-cobordism theorem for which a symplectic version is  missing.

\subsection*{Notation and conventions}
\begin{small}
 I use  boldface for vectors in $\bR^n$. Furthermore the $n$ ball of radius $r$ is denoted 
  $  
B^n(r)   =\sett{\bx\in \bR^n}{\|\bx\|<r }$ and its boundary by
$S^{n-1}(r) = \sett{\bx\in \bR^n}{\|\bx\|= r  }$.  The superscripts $n, n-1$ will be dropped if no confusion is likely to arise.
 Moreover, I'll set  $\Delta(r)=\sett{z\in \bC}{ \|z\| <r}$ and  $\Delta^*(r)=\Delta(r)\setminus \set{0}$.
 
 If $X$ is a smooth manifold, $T(X)$ is its tangent bundle and for  $x\in X$ the vector space $T_x(X)$ is the tangent space at $x$.
 \end{small}

\subsection*{Acknowledgements}
\begin{small} Thanks to Duco van Straten for pointing out the recent results \cite{boba,auy}. 
Thanks also to Federica Pasquotto and Nikolas Adaglou for several clarifying 
discussions. 
\end{small}

\section{On the topology of  isolated hypersurface singularities}
  \label{sec:top}
 
 \subsection{The Milnor fibration}
 
%\begin{small} 
%In this section  some well-known results are recalled concerning  the topology of isolated  hypersurface singularities.
%This goes back to Milnor's work \cite{milnorbook}; similiar results in the more general setting of 
% isolated singularities which are not necessarily of hypersurface type can be found in Looijenga's book~\cite{looij}
% and in the more recent volume ~\cite{seade}  by  L\^e et al.
%\end{small}

In this note, an  isolated hypersurface singularity wil be given by a polynomial function  $f: U  \to \bC$ defined
on an open neighborhood $U$ of the origin $\bo$ in $\bC^{m+1}$ such that $f(\bo)=0$ and for which $\bo$ is the only possible critical point of $f$ in $U$. The hypersurface 
$$V(f)=\sett{\bx\in U}{f(\bx)=0}
$$ is then said to have $\bo$ as an isolated singularity.
 
Milnor has shown~\cite[Corollary 2.9]{milnorbook} that there exists  $\epsilon >0$ such that the spheres $S^{2m+1}(\epsilon')$, 
$0<\epsilon'\le \epsilon $
meet $V(f)$ transversally.   Such a     sphere   $S^{2m+1}(\epsilon)$  is called a \textbf{\emph{Milnor sphere}}.  
The   number  $ \epsilon $ will be  called a \textbf{\emph{Milnor radius.}}  Because transversality is an open condition, there are  always  Milnor spheres with radius larger 
than $\epsilon$. 
Set
\[
\epsilon_f= \sup _\epsilon\set{S^{2m+1}(\epsilon) \text{ is a Milnor sphere.}}
\]
The sphere $S^{2m+1}(\epsilon)$ being a Milnor  sphere  is equivalent to  the property that  the
 (euclidean) norm  $ \|- \|^2|_{V(f)}$   has  no critical points  $\not=\bo$ in the ball whose boundary is 
$S^{2m+1}(\epsilon)$, that is,
\begin{equation}
\label{eqn:lnk}
 (\overline B^{2m+2}(\epsilon)  \setminus \set {\bo})\cap V(f)  \mapright{   \|\bx\|^2}( 0,\epsilon^2]
\end{equation} 
has no critical points Consequently  the fibers $S^{2m+1}(\epsilon')\cap V(f)$   for $0<\epsilon'\le \epsilon$  are smooth $2m-1$-dimensional manifolds.  By Ehresman's theorem~\cite{ehres} these  are   diffeomorphic to each other and represent  the \textbf{\emph{link}} of  the singularity defined by $f$,
denoted 
\begin{equation}
\label{eqn:link}
\lnk f: =  V(f )   \cap S^{2m+1}(\epsilon' ).
 \end{equation}
If  $ \epsilon_f<\infty $ there are points $\bx\in S(\epsilon_f )\cap V(f)$ for which $\|\--\|^2$ restricted to $V(f)$ are critical. Such 
points $\bx$ are called \emph{\textbf{kinks}}  for $f$. These  will play a central role later.

The fibers $f^{-1}(z)$ near $0$ lead to the \textbf{\emph{Milnor fibration}} which appears in two guises. The one originally introduced by Milnor
in \cite[\S 4]{milnorbook} is given by the differentiable map
$  f/|f| :  S^{2m+1}(\epsilon' )   \setminus   \lnk f   \to S^1 $
which in loc. cit. is shown  to be a locally trivial fibration. Its typical fiber   $\overline {\mf f } $ is called the  closed \textbf{\emph{Milnor fiber}} of $f$.
Its interior $\mf f$ figures in a natural way in the second approach, due to  L\^e (cf. \cite{le}), the one which will be used here. It is the fibration given by $f$ itself near the singular fiber of $0$.
More precisely, one defines:

\begin{dfn} Let $\epsilon$ be a Milnor radius for the   hypersurface   $V(f)$ with isolated hypersurface singularity at $\bo$
and let $r>0$ be such that there are no kinks in  ${f^{-1} \Delta^*(r)}$ .

 \begin{enumerate}[label=(\arabic*),leftmargin=*]
\item The \textbf{\emph{Milnor tube}} $\mt{f} {\epsilon,r}$ is the open subset set of $\bC^{m+1}$ given by
$$
\mt{f} {\epsilon,r}:= {B^{2m+2}(\epsilon)}   \cap  f^{-1} \Delta  (r).
$$
\item  The \textbf{\emph{horizontal boundary of the Milnor tube}}   is 
$$
\partial^{\rm hor} \mt{f} {\epsilon,r}:= S^{2m+1}(\epsilon)    \cap   f^{-1} \Delta(r).
$$
\item
Setting  $\omt{f}{ \epsilon, r} :=  \mt{f} {\epsilon,r} \setminus f^{-1}(0)\cap B^{2m+2}(\epsilon)$,
the fibration
$$
 \omt{f}{ \epsilon, r} \mapright{\, f \,}  \Delta^*(r)
$$
is called the  \textbf{\emph{Milnor fibration}}.% and $r$ the \textbf{\emph{Milnor tube-radius}}.
\end{enumerate}

\end{dfn}

The absence of critical values    of  $f$ restricted to $\omt{f}{ \epsilon, r}$  implies that
by    Ehresman's theorem 
$f$ is the projection of  a differentiably locally trivial fiber bundle over the punctured disc $\Delta^*(r)$.

%Furthermore,   the definition of a Milnor sphere implies that $V(f)=f^{-1}(0)$ meets  $ S^{2m+1}(\epsilon)$ transversally and since
% transversality is an open condition this holds also for $f^{-1}(z)$ for $\|z\|$ small enough. 
 The horizontal boundary of the Milnor tube  contains the link $\lnk f$ of the singularity. Again by  Ehresman's theorem, 
 $f:\partial^{\rm hor} T(\epsilon,r)\to \Delta(r)$ is a differentiably locally trivial fiber bundle, this time over   the entire  disc $\Delta (r)$.
This implies the following result.

\begin{thm}[\protect{\cite{le}}] \label{thm:minorLe} Let $(V(f),\bo)$ be an isolated hypersurface singularity  in  $\bC^{m+1}$. 
Then \label{thm:AltMF}
 
\begin{enumerate}[label=(\arabic*),leftmargin=*,nosep]
\item  % Restricting $f$ to  $ {B^{2m+2}(\epsilon)}   \cap  f^{-1} \Delta^ *(r)$, the complement of the fiber $f^{-1}(0)$ in the   Milnor tube,  one
The Milnor fibration is a $C^\infty$ locally   trivial  bundle,  with fibers  
  \begin{equation}\label{eqn:MilnFib}
  \mf{f,t}:= f^{-1}(t)\cap B(\epsilon),\quad t\in \Delta^*(r),
  \end{equation}
  diffeomorphic to the (open) Milnor fiber $\mf f$.  
  \item The    horizontal boundary of the Milnor tube is the total space of  a   smooth trivial fiber bundle,
  $$
  \partial^{\rm hor} \mt{f} {\epsilon,r}  \mapright{\,f\,}  \Delta(r)
  $$
  whose  fibers 
 \begin{equation}
 \label{eqn:Links}
 \lnk {f,t}:= f^{-1}(t)\cap S(\epsilon),\quad t\in  \Delta ( r)
  \end{equation} 
   are  diffeomorphic to the  link $\lnk f$.
\end{enumerate}
  \end{thm}
 The proof  for this theorem is also available in \cite{seade} where  it appears  as Proposition 6.2.9. 
 
 \medskip
 
 The topology of such   Milnor fibers has been investigated by Milnor  in \cite{milnorbook}:

 \begin{thm}\label{thm:milnfibs}
 In case   $f\in\bC[x_1,\dots,x_{m+1}]$  is a polynomial, for some number $\mu(f)\in \bZ_{\ge 0}$  the Milnor fiber  has the homotopy type of a wedge of $\mu(f)$ $m$--spheres. 
   In particular,  its middle (integral) homology $H_m(\mf  f,\bZ)$, $2m=\dim \mf  f$,  is a free $\bZ$-module of rank $\mu(f)$, 
   the \textbf{\emph{Milnor number}} of $(X,\bo)$.   In   algebraic terms
    $
  \mu(f) =\dim \jac f $,  
    where     
    $$ \displaystyle\jac f = \bC[\![x_1,\dots,x_{m+1}]\!]/ \left(\del f{x_1},\dots,\del f{x_{m+1}}\right)
    $$
    is called the  \textbf{\emph{jacobian ring}}  of $f$.  
    In particular, if  $m\ge 2$, the Milnor fiber is simply connected.
  \end{thm}

\subsection{Isolated hypersurface singularities under deformations}

\label{ssec:isodefo}

\begin{small} In this section classical results \cite{LeRam} of L\^e--Ramanujam  and \cite{Oshea2} of O'Shea are discussed. Of  main interest here are the $\mu$-constant deformations.
\end{small}

\medskip

Depending on the allowed transformations of the
surrounding manifold there are various equivalence relations which apply to  the class of  isolated hypersurface singularities. For instance 
the  isolated hypersurface singularities $(V(f),\bo)\subset (U,\bo)$, $(V(g),\bo) \subset (U',\bo))$ ($U $  and  $U'$ open subsets of $ \bC^{m+1}$),  are 
said to be  \textbf{\emph{smoothly, respectively analytically equivalent}}, if
 for some smooth,  respectively holomorphic, embedding $\phi: (U,\bo) \into (U',\bo) $ one has $f= g\comp\phi $.  
 
 It need  not be the case that a family $\set{f_s=0}$ of  isolated hypersurface singularities,  depending smoothly on $s$ yields smoothly equivalent
 singularities. Also, in case $f_s$ depends analytically on $s$, it need  not be the case that the hypersurfaces $\set{f_s=0}$ are analytically equivalent.
 \begin{exmple}
 \label{exm:fold}
Consider the so-called  $2$-dimensional hyperbolic $T_{p,q,r}$-singularity  given by:
  \[
   f_a=x^p+y^q+z^r +a xyz,  \quad a\in \bC^\times , \quad  1/p+  1/q+  1/r <1.
  \]
  For all   $a\in \bC, a\not=0$, the polynomial   $f_a=0$  has  an isolated singularity at the origin whose 
  Milnor number  is known to be  equal to     $\mu(f)=p+q+r-1$. This is the case  since by \cite[\S 8]{arn2}  the   jacobian ring   
  is spanned  by $1$ together with the monomials $x^k$, $k=1,\dots, p-1$, $y^\ell$, $\ell=1,\dots, q-1$, $z^m$, $m=1,\dots, r-1$,
together with $xyz$. For  $a=0$ the Milnor number is equal to $(p-1)(q-1)(r-1)$ 
(a basis of the jacobian ring is given by the monomials $x^\alpha y^\beta z^\gamma$, 
$0\le \alpha\le p-2$, $0\le \beta \le q-2$, $0\le \gamma\le r-2$).  This is always larger than $p+q+r-1$ since $1/p+  1/q+  1/r <1$.
Still,  the family depends  smoothly on  $a\in\bC$.  

For $a\not=0$ the Milnor fibers and links deform smoothly with $a$.
  It is classical that the parameter $a$ 
  is a modulus, i.e. the complex structure of the singularity varies with $a$; 
  the example is one of Arnold's  unimodal singularities \index{singularity!unimodal ---}as discussed in \cite[Ch. 2.3]{singbook}.
\end{exmple}

Motivated by examples of this kind one may  introduce   weaker equivalences  which hold in particular for the members 
  $\set{f_a=0,a\in \bC\setminus \set{0}}$ of the above family.

 \begin{dfn}  1. The members of a  family  $\set{f_s=0}$   of isolated singularities    depending smoothly on $s$  are called 
 \textbf{\emph{deformation equivalent}} if  their  Milnor fibers  and  their links deform smoothly with $s$.
 \\
 2.  A smooth hypersurface  $V(F(\bz,s)) \subset \bC^{m+1} \times \Delta(r) $  depending polynomially on $\bz$ and analytically
 on $s$ and such that for all $s_0\in\Delta$ the hypersurface $V(F(\bz,s_0) )\subset \bC^{m+1}$ has an isolated singularity at $\bo$
 and such  that the Milnor numbers  for these hypersurfaces  are constant
 defines the $\mu$-constant deformation $\set{f_s}_{s\in \Delta(r)}$, $f_s=V(F(-,s))$.
  The corresponding  singularities are called  \textbf{\emph{$\mu$-equivalent}}.
\end{dfn}

Let me return to the present setting of  a family  of (germs of) hypersurfaces $V(f_s)$ depending \emph{polynomially} on $s$, all having an isolated
singularity at $\bo$. Assume  for simplicity that  the neighborhood $U\subset \bC^{m+1}$ of the origin is chosen so that for
all $s\in \Delta(r)$ the hypersurface $V(f_s)\cap U$ has no other singularities besides $\bo$ and there is a critical point $\bp _s\in V(f_s)$
of the function $\|\,\, \|^2_{V(f_s)}$ nearest to the origin, by definition a  \textbf{\emph{kink}} of $f_s$.
Then a Minor radius for $f_s$ is strictly smaller than $\|\bp(s_0) \|$.
There might be a sequence of kinks $\bp_s$ with $\lim_{s\to s_0} \bp_s=\bo $ and then in a neighbourhood of $s_0$ one cannot find a common
Milnor radius. By the curve selection lemma one then finds a   \textbf{\emph{vanishing fold}} 
centered at   the singular point  $(\bo, s_0)$ of the fibers of the family. 
By definition this is a real analytic arc 
$$
\gamma  :  [0,\delta)   \to U\times \Delta(s_0,r_0)\subset \bC^{m+1}\times\Delta(r),\,
 \gamma(u)=(\bx(u), s(u))\in (V(f_{s(u)}), s(u))
$$  
 starting at the singular point $(\bo, s_0)\in V(f_{s_0})$ such that $\gamma(u)$ for $u\not=0$ is a kink of $f_{s(u)}$.
The  proof of the main result in \cite{Oshea2} implies  that in the present setting,  
 in case  $\mu(f_{s_0})>  \mu(f_{s })$ for all  $s$ in a neighborhood   of $s_0$,  there exists such a  vanishing fold centered at $(\bo, s_0)$.    
For example,   the family of Example~\ref{exm:fold} has a vanishing fold  centred at  $(\bo,0)$.

In case one  has a $\mu$-constant deformation,  vanishing folds might exist  which    is equivalent to the 
non-existence of  a common Milnor  radius   for all $f_s$:
 
\begin{obs}
\label{obs:nofolds}
Let   $V(f_s)$, $s\in \Delta$   be a family  of (germs of) hypersurfaces depending holomorphically  on $s $, all having an isolated
singularity at $\bo$ with the same Milnor number. Suppose the family has no vanishing folds. Then  the Milnor fibers
$\mf{V(f_s)}$ have a common Milnor tube radius and hence  form a differentiably locally trivial family.
In particular, any two members of the family $\set{V(f_s)}$ are  deformation-equivalent.
\end{obs}

\begin{rmk} \label{rmk:foldlessmu} 
By  \cite[Remark (3.5]{Oshea2} a $\mu$-constant family of weighted homogeneous  (or quasi-homogeneous) polynomials $f_s$ having the same weights and degree
has no vanishing folds and so Observation~\ref{obs:nofolds} applies.
\end{rmk}

 The absence of vanishing folds in arbitrary $\mu$-constant families   seems hard to show. 
 Sidestepping this problem, L\^e and Ramanujam could establish 
deformation equivalence  for a $\mu$-constant family, except if $m=2$:

\begin{thm}[\protect{\cite[Theorem 2.1]{LeRam}}] \label{thm:leram}
The (closed) Milnor fibers for  a   $\mu$-constant deformation  of an $m$-dimensional isolated singularity
  are differentiably  isomorphic, provided that $m\not=2$.  
\end{thm}

\begin{rmk}\label{rmk:hCobord}
The proof of this theorem for  the curve case uses basic facts about the topology of topological surfaces,
while for $m\ge 3$ it essentially uses Smale's  $h$-cobordism theorem~\cite{smale1,milnorcobord} to compare the Milnor fibers of $f_0$ 
with that of a   deformation $f_s$.  To apply the theorem it is essential 
 that by Theorem~\ref{thm:milnfibs}  these  Milnor fibers  in  dimension $\ge 3$  are simply connected
 and    have the same homotopy type.
\end{rmk}

 \section{Symplectic aspects}
 
 \subsection{The Milnor fibration from a symplectic point of view}
 \label{ssec:bassymp}
 
 \begin{small} In this section some  background on symplectic and  contact geometry is revised and applied to the setting of the
 Milnor fibration. The main result here is Proposition~\ref{prop:keat}, a symplectic version of the Milnor fibration  due to A. Keating~ \cite{keat}
 \end{small}
 \medskip
 
Recall that a  \textbf{\emph{symplectic manifold}} $(M,\omega)$ is a differentiable manifold of even dimension 
with a non-degenerate closed form $\omega$. A diffeomorphism of $M$ preserving $\omega$ is called a \textbf{\emph{symplectomorphism}}.
 K\"ahler manifolds are important examples of symplectic manifolds,  such as $\bC^{m+1}$ equipped with  
$\omega_{\bC^{m+1}} =\half \ii \sum_{j=1}^{m+1}  dz_j\wedge \overline{dz}_j  $ 
and complex submanifolds     $X\subset \bC^{m+1}$ equipped with $\omega_X=\omega_{\bC^{m+1}}|_{X}$.  
Observe that in this situation $\omega$ is even exact:
\begin{equation}
 \label{eqn:stsndardforms}
  \omega_{\bC^{m+1}}=  d \alpha_{\bC^{m+1}} , \quad  \alpha_{\bC^{m+1}}= \frac 14 \sum_{j=1}^{m+1}   z_j \overline{dz}_j  -
   \overline z_j {dz}_j.
  \end{equation}  
  
  \begin{exmple}[Milnor tubes and Milnor fibers] \label{ex:symp1}  For an  isolated hypersurface singularity  $(V(f), \bo)$  in 
$\bC^{m+1}$ its   Milnor fibers $\mf{f,t}=f^{-1}t \cap {\mt{f}{\epsilon,r}}$, $t\in \Delta^*(r)$    are 
  smooth $n$-dimensional complex submanifolds  of an open ball in $\bC^{m+1}$ and thus inherit  a canonical symplectic structure from
  the above one on $\bC^{n+1}$  and this holds also for 
the Milnor tube $\mt{f}{\epsilon,r}$:
\begin{equation}
\label{eqn:exsympmil}
 \omega=d\alpha:= \omega_{\bC^{m+1}} |_{\mt{f}{\epsilon,r}},\, \alpha= \alpha_{\bC^{m+1}}|_{\mt{f}{\epsilon,r}}.
\end{equation}
In this set-up the tangent space to a point  $x\in \mf{f,t}$
has a natural complement in the tangent space  in the tangent space at $x$ of the Milnor tube, namely the symplectic complement
$$
T^{\rm hor}_x(\mf{f,t} ):= \sett{ X\in T_x(\mt{f}{r,\epsilon}) } {\omega(X, Y)=0 \, \forall Y\in T_x(\mf{f,t})},
$$ 
which is called the \emph{\textbf{horizontal tangent space}} at $x$. These define a smooth (real) subbundle $T^{\rm hor}(\mf{f,t} )\subset T(\mf{f,t} )$ of rank $2$.
  \end{exmple}

A   \textbf{\emph{contact manifold}} is an odd-dimensional manifold admitting
 a contact  structure, i.e., admitting a field   of  hyperplanes in the tangent bundle  
 which defines a maximally non-integrable   distribution. There is a non-degenerate 1-form, called \textbf{\emph{contact 
 form}} on the contact manifold whose  null-space is exactly this field of hyperplanes. A contact form is not-unique: it can be multiplied with any global non-zero function. A diffeomorphism preserving the contact structure is called a \textbf{\emph{contactomorphism}}.
 
 \begin{exmple}[Links] \label{ex:lnk} 
The contact manifolds related to Example~\ref{ex:symp1}  
are the  links $\lnk {f,t}:= f^{-1}(t)\cap S(\epsilon)$, $t\in  \Delta ( r)$.
To see this, let $J$ be the standard almost complex structure on the submanifold 
$\mf{f,t}\subset \bC^{m+1}$.
Then at every point $x\in\lnk {f,t}$ the intersection $T_x \lnk {f,t}\cap J(T_x\lnk {f,t})$
gives a contact structure with $\alpha|_{\lnk {f,t}}$  a contact form.
\end{exmple}  
  
 The   closures of the Milnor fibers give    examples  of   Liouville domains as will be shown below in Example~\ref{exm:ldoms}.
   By definition   a \textbf{\emph{Liouville domain}} is a compact symplectic manifold $(W,\omega)$   with a smooth
 compact boundary $S=\partial W $   admitting   Liouville field $\bY$ 
  which is transverse to $S$.
 A  \textbf{\emph{Liouville field}} $ \bY$ on a symplectic manifold $(W,\omega)$ 
 by definition is a smooth vector field which preserves the form $\omega$, that is, $\cL _\bY \omega=\omega$, where $\cL_\bY$ is the Lie-derivative. 
 In case $\omega=d\alpha$ this is the same as saying that $\iota_\bY (\omega)= \alpha$. In that case $\bY$ is uniquely defined by the choice of $\alpha$ for which $d\alpha=\omega$.
 Such Liouville domains  $(W,\omega)$ are called exact Liouvile domains. 
These  can be embedded in a larger symplectic manifold
  $(\widehat W, \widehat\omega)$,  its \textbf{\emph{symplectic completion}} by
  using  the symplectic cylinder  $\cyl S:=(S\times \bR , d(e^x\alpha_S))$ on $(S,\alpha_S)$
  whose  Liouville field is $ d/dx$.  In fact, one glues $W$ to  
  $\cyl S_{\ge 0} =(S\times \bR_{\ge 0}, d(e^x\alpha_S))$ along $S$ which is possible using
  the flow of the Liouville field because it gives a  suitable  neighborhood $U$ of $S$ in $W$  a
 cylinder-like structure, say   $  S\times [-\delta,0] \mapright{\sim} U\subset   W$ so that $U\cup_S \cyl S_{\ge 0}$ 
 glues to $W$ and gives the desires symplectic completion $\widehat W$. 
    
\begin{exmple}[Closed Milnor fibers] \label{exm:ldoms} Since on the closure of the Milnor fiber $\omega=d\alpha$, there is a unique Liouville field.
The corresponding Liouville fields   $\bY_t$ on each fiber  glue to a Liouville field $\bY$ on the Milnor tube away from the fiber over $0$.
One can  extend this  field to the horizontal boundary 
$ \partial^{\rm hor} \mt{f}{\epsilon,r}$,   the union of the links $\partial \mf{f,t}$.
  This  is always possible by considering this  boundary
inside a slightly larger Milnor tube  $\mt{f} {\epsilon',r}$, $\epsilon'>\epsilon$. Then  one may view  the collar 
$\mf{f,t}\cap \sett{\bz\in \bC^{n+1}}{\epsilon\le \| \bz \|<\epsilon'}$ as a copy of the symplectic cylinder and their union then is 
 \[
 \cmt{f}{\epsilon',r} = \bigcup _{t\in \Delta^*(r)} \widehat{ \mf{f,t} }\text{\textbf{\emph{ the completed Milnor tube}}}.
 \]
\end{exmple}

 The Milnor fibration   $f: \omt{f}{ \epsilon, r}   \to \Delta^*(r)$ is an example of a \textbf{\emph{symplectic fibration between
symplectic manifolds}}  where the symplectic structure of the fibers is induced by the form $\omega$  
on the Milnor tube.\footnote{A symplectic fibration  is defined as a smooth   fibration  between differentiable manifolds  such that the fibers are symplectic manifolds 
and the transition functions are symplectomorphisms. In particular the total space need not be symplectic.
 See e.g. \cite[Example 6.5]{SYmpTop}.}  
This form is a connection form for the fibration since it defines a unique splitting 
$T (\omt{f}{\epsilon,r} |_{\mf{f,t}}= T(\mf{f,t} )\oplus T^{\rm vert}(\omt{f}{\epsilon,r})|_{\mf{f,t}}$.
The vertical tangent bundle allows parallel transport  of a given Milnor fiber along a path $\gamma$ in $\Delta^*(r)$. 
 By construction  parallel transport using $\omega$
  preserves the symplectic form on the Milnor fibers. However, the vertically lifted tangent field of $\gamma$ might  not be everywhere defined preventing
 the Milnor tube from being a symplectically locally trivial fiber bundle. 
  The idea from Keating's article~\cite{keat}  is to make a 
  \begin{modif} \label{mod:crux}
  Replace $\omega'=d\alpha'$  in such a way that 
\begin{enumerate}[label=(\arabic*),leftmargin=*,nosep]
 \item $\omega'$ and $\omega$ coincide outside a  collar neighborhood  of the horizontal boundary $\partial^{\rm hor} \mt{f}{\epsilon,r}$ of
 the Milnor tube.
 \item the horizontal tangent spaces defined by the connection form $\omega'$ are preserved by the (vertical)
 Liouville flow $\bY$ on the Milnor tube.
\end{enumerate}
\end{modif}

\noindent The first property shows that the horizontal tangent space  gets adapted on a collar neighborhood of 
 $\partial^{\rm hor} \mt{f}{\epsilon,r}$, while 2) shows that the Liouville flow preserves
the  horizontal tangent spaces with respect to $\omega'$ along the Milnor fibers. These properties guarantee  that these give  smoothly varying 
 horizontal tangent spaces  along the completed Milnor fibers.
%\medskip
%
%\noindent The construction of $\omega'$ is done in two steps. 
% \\
%(1) One uses the flow $\Psi^{-\bY_f}_u$, $u\in [0,\delta]$ of the Liouville field
%to give a fiber preserving diffeomorphism of $\partial^{\rm hor} T(f)\times [-\delta,0] $ with a neighborhood $U\subset \mt{f}{\epsilon,r}$ of $\partial^{\rm hor} \mt{f}{\epsilon,r} $ and one defines
%$\alpha'_U =e^t  \cdot (\phi^{-\bY_f})^*\alpha_U$ which restricts on each fiber to $\alpha_{U\cap f^{-1}t}$ since
%  there $[\Psi^{-\bY_f}_t]^*\alpha= e^{-t} \alpha$. However the global form  $\alpha'_U$ is in general different from $\alpha_U$.
%  This identification can  be used to glue  $U$ to $ \partial^{\rm hor} \mt{f}{\epsilon,r}\times [0,\infty] $.
%  The form $d\alpha'$ has property 2) on $U$.
% \\
%(2)  Choose a smooth cut-off function    $f:[-\delta,0]\to [0,1]$ which is $1$ near the right end and $0$ near the left end. 
%This als defines $f:U\to [0,1]$ which takes the value 1  on   the collar neighborhood $f^{-1}(1)$ of  $\partial^{\rm hor} 
% \mt{f}{\epsilon,r}$ and which  is 0 on a neighborhood of  $\partial^{\rm hor} \mt{f}{\epsilon,r} \times\set{-\delta }$.
% Then 
% \[
% \omega'= \omega+ d(f (\alpha'-\alpha)) 
% \]
% still has property 2), but it now also has property 1).
%This then completes the proof of the following proposition.
%    
%
 The result is then as follows.   
      
 \begin{prop}[\protect{\cite[\S9.1]{keat}}]  \label{prop:keat} Let $(V(f) ,\bo)\subset (\bC^{m+1},\bo)$ be an isolated hypersurface singularity and $B(\epsilon )$ an associated Milnor ball. Let $\omega'=d\alpha'$ be as in the above modification~\ref{mod:crux}.
 Then  for     $\epsilon<\epsilon'  $ small enough Milnor radii  and for small enough $r>0$, the fibration 
 $$
 \hat f: \cmt{f}{\epsilon',r} \to \Delta^*(r)
 $$
 whose total space is the completed Milnor tube,   is a locally trivial symplectic fibration with respect to $\omega'$.
 Moreover, the forms $ \alpha'_{ \lnk {f,t} }$  on  the links  $\lnk {f,t}$, $t\in \Delta(r)$  
 induce  contactomorphic contact  structures.
         \label{prop:keating} 
\end{prop}
 
\begin{rmk} \label{rmk:onsympfib} (1) In \cite[Example 5.22]{boba} there is a construction of a symplectic   fibration 
$f|_N: N \to \Delta^*(r)$, where $N$ is a subset of a suitable Milnor ball $B(\epsilon)$ and $r$ is small enough.
Although $N$ is not constructed from the symplectic completions of the Milnor fibers, the fibration $f|_N$ is 
diffeomorphic to the Milnor fibration $f$. This fibration    behaves well for all
$\mu$-constant deformations as  explained  in loc. cit, Example 5.23, although one can only control the $C^\infty$-structure of the Milnor
fibers of the deformation using Theorem~\ref{thm:leram}.
\\
(2) A  priori  its not  clear that   the symplectomorphism class   of $\mf {f,t}$  itself does not change when one varies $t$ in $\Delta^*$.
However, since the completed Milnor fibers   $\widehat{\mf{f,t}}$ vary  in a symplectic fiber  bundle, 
its symplectomorphism class is constant.
This also holds for the contactomorphism class of the link.  
\end{rmk}

\subsection{$\mu$-equivalence in the symplectic setting}

 \label{ssec:muequiv}
 
The main result  of this note is as follows.

 \begin{thm}\label{thm:MuConstResult}
 Any two (completed) Milnor fibers of
  a $\mu$-constant polynomial deformation  without vanishing folds are symplectomorphic  and the links of the corresponding singularities are contactomorphic.
  \end{thm}

\begin{proof} 
 Consider a  1-parameter deformation  of   isolated hypersurface singularities 
 $(V(f_s),\bo)$, say $f_s\in \bC[z_0,\dots,z_m,s] $, $s\in  {\Delta(\delta)}$.
Since by Observation~\ref{obs:nofolds} there are  no vanishing folds in  the family,  
there is  a     Milnor  radius  $\epsilon$   which is common to all $f_s$. The set of critical values of $f_s$ is closed and contains 
$\bo \times \Delta(\delta)$. Hence one can find a common tube radius $r>0$ such that $F((z_0,\dots,z_m),s)=f_s(z_0,\dots,z_m) $ is regular on 
$$
\omt{F} {\epsilon,r} :=\left[  F^{-1}( \Delta^*(r)\times \Delta(\rho) )\cap S^{m+1}(\epsilon)\right]\times \Delta(\rho) ,
$$
the total space of the fibration $\omt{F}{\epsilon,r}  \mapright{ F } \Delta^*(r)\times\Delta(\delta)$. 
The exact symplectic form which one takes on the left hand side is the restriction of $\widetilde\omega=\omega+ \frac 14 d\left(sd\bar s-\bar s ds \right)$,
where $\omega$ is as in Eqn.~\ref{eqn:exsympmil}.
Replacing the horizontal boundary of the Milnor tube
by  $\bigcup_{s\in \Delta(\rho)} \partial^{\rm hor} \omt{f_s} {\epsilon,r}$ and adapting  $\widetilde\omega$ in a collar 
$\left[(\epsilon\le \|\bz\|<\epsilon')\times \Delta(\rho)\right]\cap \omt{F}{\epsilon',r}$,
 as in  the proof of Proposition~\ref{prop:keating} sketched in the previous section, one obtains
a locally trivial symplectic fibration   
$$ 
 \cmt{F} {\epsilon',r} =\bigcup_{s\in \Delta(\rho)}  \cmt{f_s} {\epsilon',r}  \mapright{ F(\bz,s)}  \Delta^*(r)\times\Delta(\delta).
$$
 This completes the proof.
 \end{proof}

Remark~\ref{rmk:foldlessmu}   implies the following result.

 \begin{corr}  \label{cor:MuConst} Any two (completed) Milnor fibers of
  a $\mu$-constant polynomial deformation  of an  isolated  weighted homogeneous hypersurface  
  are symplectomorphic  and the links of the corresponding singularities are contactomorphic.

\end{corr}

 \bibliographystyle{alpha}

 \end{document}